\documentclass[10pt]{amsart}

\usepackage{amsmath,amsfonts,latexsym,amssymb,mathrsfs,setspace,enumerate,color,cite,graphicx,epsf}

\newtheorem{thm}{Theorem}
\newtheorem{lem}[thm]{Lemma}

\newtheorem{cor}[thm]{Corollary}
\numberwithin{equation}{section}
\numberwithin{thm}{section}

\newtheorem{conj}[thm]{Conjecture}

\newcommand{\rat}{\mathbb Q}
\newcommand{\real}{\mathbb R}
\newcommand{\com}{\mathbb C}
\newcommand{\alg}{\overline\rat}
\newcommand{\algt}{\alg^{\times}}

\newcommand{\intg}{\mathbb Z}
\newcommand{\nat}{\mathbb N}

\newcommand{\X}{\mathcal X}

\newcommand{\tor}{\mathrm{Tor}}

\newcommand{\rad}{\mathrm{Rad}}

\newcommand{\xx}{{\bf x}}
\newcommand{\yy}{{\bf y}}
\newcommand{\mm}{{\bf m}}
\newcommand{\XX}{\mathcal X}
\newcommand{\bigreal}{\X(\real)}

\newcommand{\comment}[1]{}

\title{The parametrized family of metric Mahler measures}
\author[C.L. Samuels]{Charles L. Samuels}
\thanks{This research was supported in part by NSERC of Canada} 

\begin{document}

\maketitle

\begin{abstract}
	Let $M(\alpha)$ denote the (logarithmic) Mahler measure of the algebraic number $\alpha$.  Dubickas and Smyth, and later Fili and the author, examined metric versions of
        $M$. The author generalized these constructions in order to associate, to each point in $t\in (0,\infty]$, a metric version $M_t$ of the Mahler measure, 
        each having a triangle inequality of a different strength.  We further examine the functions $M_t$, using them to present an equivalent form
        of Lehmer's conjecture.  We show that the function $t\mapsto M_t(\alpha)^t$ is constructed piecewise from certain sums of exponential functions.  We pose a 
        conjecture that, if true, enables us to graph $t\mapsto M_t(\alpha)$ for rational $\alpha$.
\end{abstract}

\section{Introduction}

Let $f$ be a polynomial with complex coefficients given by
\begin{equation*}
	f(z) = a\cdot\prod_{n=1}^N(z-\alpha_n).
\end{equation*}
We define the {\it (logarithmic) Mahler measure} $M$ of $f$ by
\begin{equation*}
	M(f) = \log|a|+\sum_{n=1}^N\log^+|\alpha_n|.
\end{equation*}
If $\alpha$ is a non-zero algebraic number, we define the {\it (logarithmic) Mahler measure $M(\alpha)$ of $\alpha$} to be the Mahler measure of the minimal 
polynomial of $\alpha$ over $\intg$.

It is a consequence of a theorem of Kronecker that $M(\alpha) = 0$ if and only if $\alpha$ is a root of unity.  In a famous 1933 paper, D.H. Lehmer \cite{Lehmer} 
asked whether there exists a constant $c>0$ such that $M(\alpha) \geq c$ in all other cases.  He could find no algebraic number with Mahler measure smaller than 
that of
\begin{equation*}
  \ell(x) = x^{10}+x^9-x^7-x^6-x^5-x^4-x^3+x+1,
\end{equation*}
which is approximately $0.16\ldots$.  Although the best known general lower bound is
\begin{equation*}
	M(\alpha) \gg \left(\frac{\log\log\deg\alpha}{\log\deg\alpha}\right)^3,
\end{equation*}
due to Dobrowolski \cite{Dobrowolski},
uniform lower bounds have been established in many special cases (see \cite{BDM, Schinzel, Smyth}, for instance).  Furthermore, numerical evidence
provided, for example, in \cite{Moss, MossWeb, MPV, MRW} suggests there does, in fact, exist such a constant $c$.  
This leads to the following conjecture, which we will now call Lehmer's conjecture.

\begin{conj}[Lehmer's conjecture] \label{LehmerConjecture}
	There exists a real number $c > 0$ such that if $\alpha\in\algt$ is not a root of unity then $M(\alpha) \geq c$.
\end{conj}

Dubickas and Smyth \cite{DubSmyth2}, and later Fili and the author \cite{FiliSamuels}, examined metric and ultrametric versions of the
Mahler measure on $\alg$, respectively.  In \cite{Samuels3}, we noted that these constructions arise from the following more general principle.

Let $G$ be an abelian group (written multiplicatively) with identity $e$.  We say that $\phi:G \to [0,\infty)$ is a {\it (logarithmic) height} on $G$
if the following two conditions are satisfied.
\begin{enumerate}[(i)]
	\item $\phi(e) = 0$,
	\item $\phi(\alpha) = \phi(\alpha^{-1})$ for all $\alpha\in G$.
\end{enumerate}
If $\psi$ is another height on $G$, we follow the conventional notation that
\begin{equation*}
	\phi = \psi \quad \mathrm{or} \quad \phi \leq \psi
\end{equation*}
when $\phi(\alpha) = \psi(\alpha)$ or $\phi(\alpha) \leq \psi(\alpha)$ for all $\alpha\in G$, respectively.  We write
\begin{equation*}
	Z(\phi) = \{ \alpha\in G: \phi(\alpha) = 0\}
\end{equation*}
to denote the {\it zero set} of $\phi$.

If $t$ is a positive real number then we say that $\phi$ has the {\it $t$-triangle inequality} if
\begin{equation} \label{tTriangleInequality}
	\phi(\alpha\beta)^t \leq \phi(\alpha)^t + \phi(\beta)^t
\end{equation}
for all $\alpha,\beta\in G$.  We say that $\phi$ has the {\it $\infty$-triangle inequaltiy} if
\begin{equation} \label{InftyTriangleInequality}
	\phi(\alpha\beta) \leq \max\{\phi(\alpha),\phi(\beta)\}
\end{equation}
for all $\alpha,\beta\in G$.   We observe that the $1$-triangle inequality is simply the classical triangle inequality while the $\infty$-triangle
inequality is the strong triangle inequality.
A height $\phi$ satisfying \eqref{tTriangleInequality} or \eqref{InftyTriangleInequality} is called a {\it $t$-metric height}
or {\it $\infty$-metric height}, respectively.  It is noted in \cite{Samuels3} that such heights have the following properties.
\begin{enumerate}[(i)]
\item\label{Subgroup} $Z(\phi)$ is a subgroup of $G$.
\item\label{WellDefined} $\phi$ is well-defined on the quotient $G/Z(\phi)$.
\item\label{FancyMetric} If $t\geq 1$, then the map $(\alpha,\beta)\mapsto \phi(\alpha\beta^{-1})$ defines a metric on $G/Z(\phi)$.
\end{enumerate}

If $\phi$ is a height which is not necessarily a $t$-metric height, then we may construct a natural $t$-metric version of $\phi$.  For simplicity, we will now write
\begin{equation*}
	\XX(G) = \{(\alpha_1,\alpha_2,\ldots): \alpha_n\in G\ \mathrm{and}\ \alpha_n=e\ \mathrm{for\ all\ but\ finitely\ many}\ n\}.
\end{equation*}
If $\real$ denotes the group of real numbers under addition, $\xx = (x_1,x_2,\cdots) \in\bigreal$, and $t$ is any positive real number, we define 
\begin{equation} \label{tNorm}
	\|\xx\|_t = \left(\sum_{n=1}^\infty |x_n|^t\right)^{1/t}\quad\mathrm{and}\quad \|\xx\|_\infty = \max_{1\leq n}\{|x_n|\}.
\end{equation}
In the case where $t\geq 1$, we know that $\|\xx\|_t$ is the $L^t$ norm of $\xx$.  If $t<1$, then \eqref{tNorm} does not define a norm on $\bigreal$, 
but we continue to use the same notation for the sake of consistency.  Let $\tau:\XX(G) \to G$ be defined by 
\begin{equation*}
	\tau(\alpha_1,\alpha_2,\cdots) = \prod_{n=1}^\infty \alpha_n
\end{equation*} 
and note that $\tau$ is a group homomorphism.  The {\it $t$-metric version of $\phi$} is given by
\begin{equation*}
	\phi_t(\alpha) = \inf\left\{\| (\phi(\alpha_1),\phi(\alpha_2),\ldots) \|_t: (\alpha_1,\alpha_2,\ldots)\in\tau^{-1}(\alpha)  \right\}
\end{equation*}
so that the infimum is taken over all ways of writing $\alpha$ as a product of elements in $G$.  It is immediately clear that if $\psi$ is another 
height on $G$ with $\phi \geq \psi$, then $\phi_t \geq \psi_t$ for all $t$.  The results of \cite{Samuels3} establish the following additional observations.
\begin{enumerate}[(i)]
\item\label{MetricHeightConversion} $\phi_t$ is a $t$-metric height on $G$ with $\phi_t\leq\phi$.
\item\label{BestMetricHeight} If $\psi$ is an $t$-metric height with $\psi\leq\phi$ then $\psi\leq \phi_t$.
\item\label{NoChangeMetric} $\phi = \phi_t$ if and only if $\phi$ is an $t$-metric height.
\item\label{Comparisons} If $s\in (0,t]$ then $\phi_s \geq \phi_t$.
\end{enumerate}

It is well-known that the Mahler measure $M$ is a height on $\algt$ with $Z(M)$ equal to the set of roots of unity.   It follows from the results of \cite{DubSmyth2} and
\cite{FiliSamuels} that $Z(M_t) = Z(M)$ for all $t\in (0,\infty]$.  Among other things, it is noted that $M_1$ and $M_\infty$ induce
the discrete topology on $$V = \algt/Z(M)$$ if and only if Lehmer's conjecture is true.  It turns out that we have something stronger.

\begin{thm} \label{Topologies}
	Lehmer's conjecture is true if and only if there exists $t \in [1,\infty)$ such that $M_t$ and $M_\infty$ induce the same topology on $V$.
\end{thm}

Our goal for the remainder of this article is to examine the functions $t\mapsto M_t(\alpha)$ for a fixed algebraic number $\alpha$.  For simplicity, we define
$\mu_\alpha: (0,\infty] \to [0,\infty)$ by
\begin{equation*} \label{MetricFunctionDef}
	\mu_\alpha(t) = M_t(\alpha).
\end{equation*}   
It is clear from our earlier remarks that $\mu_\alpha$ is decreasing, bounded above by $M(\alpha)$, and
$\mu_\alpha(t)$ tends to $M_\infty(\alpha)$ as $t\to \infty$.  The results of \cite{Samuels3} give some additional properties of $\mu_\alpha$, namely
\begin{enumerate}[(i)]
\item $\mu_\alpha$ is continuous on $(0,\infty)$,
\item $\mu_\alpha$ is constant in a neighborhood of $0$, and
\item\label{AttainedList} The infimum in the definition of $\mu_\alpha(t)$ is always attained.
\end{enumerate}
This final observation suggests the following direction of study.  While the set
\begin{equation*}
	A_\alpha(t) = \{\xx\in\bigreal: \mu_\alpha(t) = \|\xx\|_t\}
\end{equation*}
is always non-empty, it is possible that $A_\alpha(t_1) \cap A_\alpha(t_2)$ is empty for different points $t_1$ and $t_2$.
This suggests that there are points $t\in (0,\infty)$ such that the point $\xx$ where the infimum is attained must change.
We call these points {\it $\alpha$-exceptional} and capture this concept rigorously in the following way.  

A set $I\subseteq (0,\infty]$ is called {\it $\alpha$-uniform} if there exists a point $\xx\in\bigreal$ such that
\begin{equation*}
	\mu_\alpha(t) = \|\xx\|_t
\end{equation*}
for all $t\in I$.  A point $s\in (0,\infty]$ is called {\it $\alpha$-standard} if there exists an $\alpha$-uniform open neighborhood of $s$.  
If $s$ is not $\alpha$-standard, then we say that $s$ is {\it $\alpha$-exceptional}.  Our first result shows that the set of $\alpha$-exceptional
points is rather sparse.

\begin{thm} \label{FiniteBelow2}
 	If $\alpha$ is a non-zero algebraic number and $T$ is a positive real number, then there are only finitely many $\alpha$-exceptional points in $(0,T)$.
\end{thm}

It is an open question to determine whether there are only finitely many $\alpha$-exceptional points in all of $(0,\infty)$.  The proof of Theorem \ref{FiniteBelow2} relies on 
an upper bound, depending on both $\alpha$ and $T$, on the number of terms that may appear in any factorization of $\alpha$.  It appears that we cannot remove the
dependency on $T$ to establish the finiteness of the set of $\alpha$-exceptional points.  Nonetheless, we know of no example of an algebraic number $\alpha$ having
infinitely many $\alpha$-exceptional points.

Conceptually, the $\alpha$-exceptional points represent values of $t$ at which the infimum attaining point $\xx$ must change.
Our next Theorem shows that the intervals between the $\alpha$-exceptional points contain no such changes.

\begin{thm} \label{UniformIntervals}
	Suppose that $0 < a < b <\infty$.  Then $[a,b]$ is $\alpha$-uniform if and only if every point in $(a,b)$ is $\alpha$-standard.
        Moreover, $(0,a]$ is $\alpha$-uniform if and only if every point in $(0,a)$ is $\alpha$-standard.
\end{thm}

We now apply Theorems \ref{FiniteBelow2} and \ref{UniformIntervals} to show that $\mu_\alpha$ may be constructed piecewise
from functions of the form $t\mapsto \|\xx\|_t$.  The pieces are divided precisely by the $\alpha$-exceptional points.

\begin{cor} \label{NonOver}
Let $\alpha$ be a non-zero algebraic number and $T$ a positive real
number.  There exists a finite collection of non-overlapping intervals $\mathcal I$, each closed in $(0,T]$, such that
\begin{enumerate}[(i)]
	\item\label{Uniforms} Each interval in $\mathcal I$ is $\alpha$-uniform,
	\item\label{Union} $\displaystyle (0,T] = \cup_{I\in \mathcal I} I$, and
	\item\label{Exceptionals} If $t\in (0,T)$ then $t$ is $\alpha$-exceptional if and only if there exist distinct intervals $I_1,I_2\in\mathcal I$ such that $t\in I_1\cap I_2$.
\end{enumerate}
\end{cor}

We now wish to establish a connection between the $\alpha$-standard points and the differentiability of $\mu_\alpha$.  Although it is clear that
$\mu_\alpha$ is infinitely differentiable at all $\alpha$-standard points, it is not obvious what happens at $\alpha$-exceptional points.  Our next theorem
gives some additional insight.

\begin{thm} \label{StandardDiff}
  Let $\alpha$ be an algebraic number and $s\in (0,\infty)$.  Then $s$ is $\alpha$-standard if and only if $\mu_\alpha$ is infinitely differentiable at $s$.
\end{thm}


\section{A conjecture on the infimum in $M_t(\alpha)$ and some applications} \label{ConjApps} 

For this section, we restrict our attention to the case that $\alpha$ is rational.  In this simpler setting, we may be able to give a more thorough description
of $\mu_\alpha$.

Recall that Theorem \ref{Achieved} shows the infimum in the definition of $M_t(\alpha)$ to be attained.  Moreover, in the case 
that $\alpha$ is rational, this infimum must be attained by a point $(\alpha_1,\ldots,\alpha_N)$ where each $\alpha_n$ is a
surd.  However, we are unable to construct an example where the infimum is not attained by a point having only rational coordinates.
This leads to the following conjecture.

\begin{conj} \label{InfimumRational}
Suppose $\alpha$ is a rational number and $t\in (0,\infty]$.  Then there exist rational points $\alpha_1,\ldots,\alpha_N$ such that
\begin{equation*}
  M_t(\alpha)^t = \sum_{n=1}^N M(\alpha_n)^t.
\end{equation*}
\end{conj}

In view of the results of \cite{DubSmyth2} and \cite{FiliSamuels}, Conjecture \ref{InfimumRational} is true for the cases
$t\leq 1$ and $t=\infty$.  In fact, in each case, a specific representation can be given that attains the infimum in
$M_t(\alpha)$.  Unfortunately, the proofs seem to be genuinely different and cannot be modified to include the intermediate values of $t$.

If Conjecture \ref{InfimumRational} is true, then we may often explicitly graph $\mu_\alpha(t)$.  Our procedure
relies on the following observation.

\begin{thm} \label{RationalFactorization}
  Suppose that $r$ and $s$ are relatively prime positive integers.  If Conjecture \ref{InfimumRational} holds, then there exist positive
  integers $r_1,\ldots,r_N, s_1,\ldots s_N$ such that
  \begin{equation*}
    M_t \left(\frac{r}{s}\right)^t = \sum_{n=1}^N M\left(\frac{r_n}{s_n}\right)^t
  \end{equation*}
 and
  \begin{equation*}
    r = \prod_{n=1}^N r_n\quad\mathrm{and}\quad s = \prod_{n=1}^N s_n.
  \end{equation*}
\end{thm}

The first statement of Theorem \ref{RationalFactorization} is simply a rephrasing of Conjecture \ref{InfimumRational}.  The real
content of the result occurs in the second statement, which shows that we need only consider all possible factorizations of the numerator and
denominator.  This allows us to determine $M_t(\alpha)$ with a finite search.  The case where $\alpha\in\intg$ is particularly straightforward.

\begin{thm} \label{Integers}
  Suppose that $\alpha$ is a positive integer and write
  \begin{equation*}
    \alpha = \prod_{n=1}^N p_n
  \end{equation*}
  where $p_n$ are not necessarily distinct primes.  If Conjecture \ref{InfimumRational} holds then
  \begin{equation*}
    M_t(\alpha)^t = \left\{ \begin{array}{ll}
        (\log \alpha)^t & \mathrm{if}\ t\leq 1 \\
        \sum_{n=1}^N (\log p_n)^t & \mathrm{if}\ t \geq 1.
        \end{array}\right.
    \end{equation*}
\end{thm}

Theorem \ref{Integers} shows, in particular, that under Conjecture \ref{InfimumRational}, an integer has no exceptional points except
possibly at $1$.  An integer has an exceptional point at $1$ if and only if that integer is composite.

It is natural to ask whether a result analogous to Theorem \ref{Integers} holds for any rational number $\alpha$.  Although we always have that
$M_t(\alpha) = M(\alpha)$ for $t\leq 1$, the situation seems to be more complicated for larger values of $t$.  We continue to assume
Conjecture \ref{InfimumRational} in the remarks that follow.

Consider, for example, $\alpha = 7/30$.  In the left column of Table \ref{fig:Factorizations}, we give all
possible representations of $7/30$ that satisfy the conclusion of Theorem \ref{RationalFactorization}.  In the right
column, we write their corresponding (non-logarithmic) Mahler measures.

\begin{table}
\caption{Factorizations of $7/30$}\label{fig:Factorizations}
\centering
\begin{tabular}{ c | c }
\hline\hline
Factorization of $7/30$ & Corresponding (non-logarithmic) Mahler measures \\
\hline
$\frac{7}{30}$ & (30) \\
$7\cdot\frac{1}{30}$ & (7,30) \\
$\frac{7}{2}\cdot \frac{1}{15}$ & (7,15) \\
$\frac{1}{2}\cdot \frac{7}{15}$ & (2,15) \\
$\frac{7}{3}\cdot \frac{1}{10}$ & (7,10) \\
$\frac{1}{3}\cdot \frac{7}{10}$ & (3,10) \\
$\frac{7}{6}\cdot \frac{1}{5}$ & (7,5) \\
$\frac{1}{6}\cdot \frac{7}{5}$ & (6,7) \\
$\frac{7}{2}\cdot \frac{1}{3}\cdot \frac{1}{5}$ & (7,3,5) \\
$\frac{1}{2}\cdot \frac{7}{3}\cdot \frac{1}{5}$ & (2,7,5) \\
$\frac{1}{2}\cdot \frac{1}{3}\cdot \frac{7}{5}$ & (2,3,7) \\
$\frac{1}{2}\cdot \frac{1}{15}\cdot 7$ & (2,15,7) \\
$\frac{1}{3}\cdot \frac{1}{10}\cdot 7$ & (3,10,7) \\
$\frac{1}{6}\cdot \frac{1}{5}\cdot 7$ & (6,5,7) \\
$\frac{1}{2}\cdot \frac{1}{3}\cdot \frac{1}{5} \cdot 7$ & (2,3,5,7)
\end{tabular}
\end{table}

We obtain immediately a natural partial ordering on the $N$-tuples $(a_1,\ldots,a_N)$ appearing in the right column of Table \ref{fig:Factorizations}.
We say that $(a_1,\ldots,a_N) \leq (b_1,\ldots,b_M)$ if
\begin{equation*}
 \|(a_1,\ldots,a_N)\|_t \leq \|(b_1,\ldots,b_M)\|_t
\end{equation*}
for all $t > 0$.  For example, we note that $(2,3,7) \leq (2,5,7)$.  On the other hand, the $L^t$ norms of $(30)$ and $(7,15)$ cross when
\begin{equation*}
  (\log 30)^t = (\log 7)^t + (\log 15)^t
\end{equation*}
so that these elements are not comparable.
An $N$-tuple $(a_1,\ldots,a_N)$ is called {\it minimal} if there does not exist another $M$-tuple $(b_1,\ldots,b_M)$ in right column of 
Table \ref{fig:Factorizations} such that $(b_1,\ldots,b_M) \leq (a_1,\ldots,a_N)$.  When computing $M_t(\alpha)$ we need only consider the minimal $N$-tuples.  
In our case, the minimal $N$-tuples are
\begin{equation*}
   (30)\quad (2,15)\quad (3,10)\quad (7,5)\quad\mathrm{and}\quad(2,3,7).
\end{equation*}
Therefore, it makes sense to define the functions
\begin{align} \label{MinFunctions}
  & f_{1}(t) = \log 30 \nonumber \\
  & f_{2}(t) = \left((\log 2)^t + (\log 15)^t\right)^{1/t} \nonumber \\
  & f_3(t) =  \left((\log 3)^t + (\log 10)^t\right)^{1/t} \\
  & f_4(t) =  \left((\log 7)^t + (\log 5)^t\right)^{1/t} \nonumber \\
  & f_5(t) =  \left((\log 2)^t + (\log 3)^t + (\log 7)^t\right)^{1/t}\nonumber
\end{align}
and note that
\begin{equation} \label{MinFunctionsCombined}
 \mu_{7/30}(t) = \min\{f_n(t): 1\leq n\leq 5\}.
\end{equation}
The graphs of the functions \eqref{MinFunctions} are given in Figure \ref{fig:graphs}. Note that we appear to have an exceptional point at
$1$ and another exceptional point $t$ satisfying the equation
\begin{equation*}
  \left((\log 10)^t + (\log 3)^t\right)^{1/t} = \left((\log 7)^t + (\log 3)^t + (\log 2)^t\right)^{1/t}.  
\end{equation*}
The apparent graph of \eqref{MinFunctionsCombined} is given in Figure \ref{fig:graphs2}.

\begin{figure}[htp]
\centering
\includegraphics[totalheight=0.4\textheight,width=\textwidth]{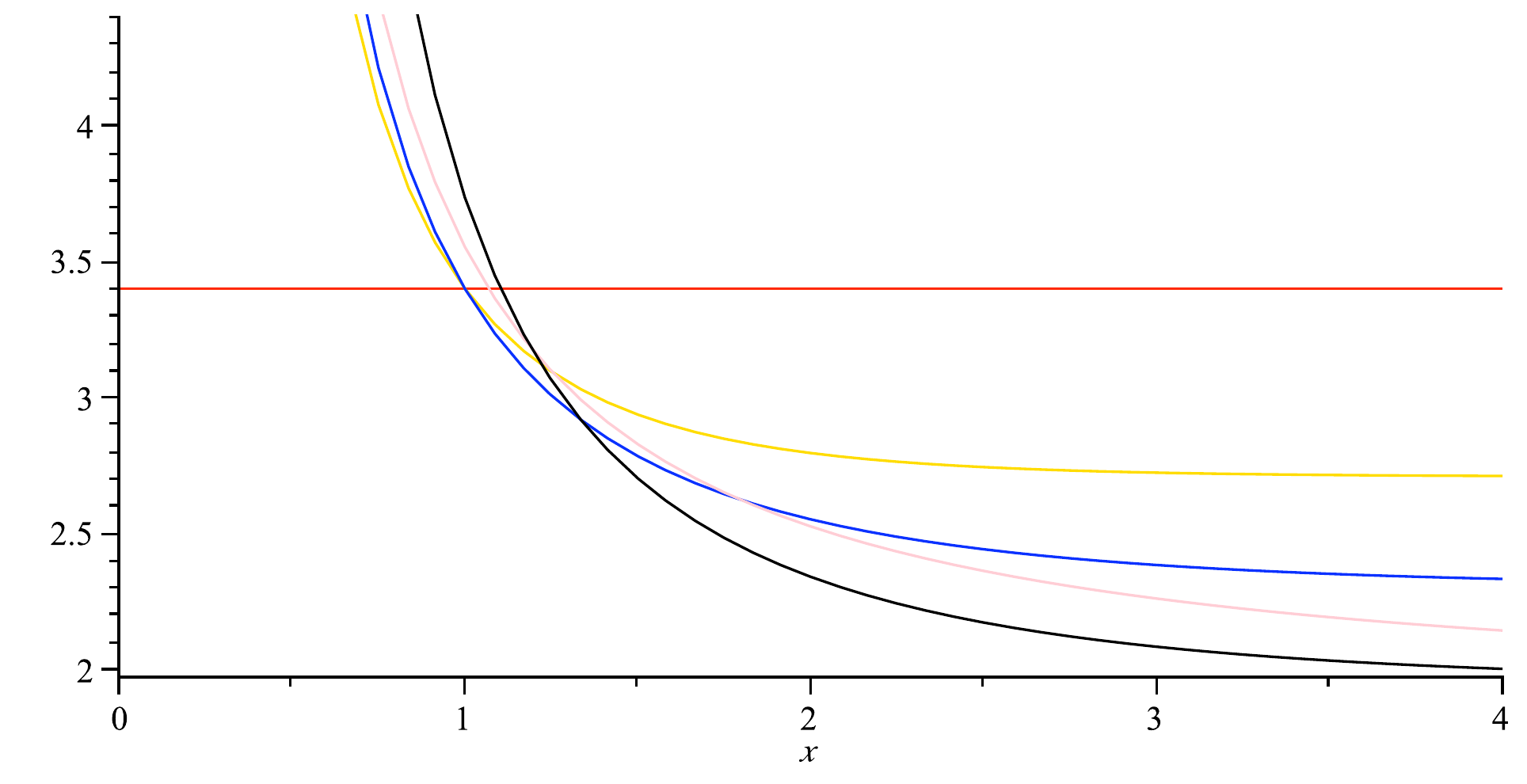}
\caption[]{Graphs corresponding to minimal representations of $7/30$}\label{fig:graphs}
\end{figure}

\begin{figure}[htp]
\centering
\includegraphics[totalheight=0.4\textheight,width=\textwidth]{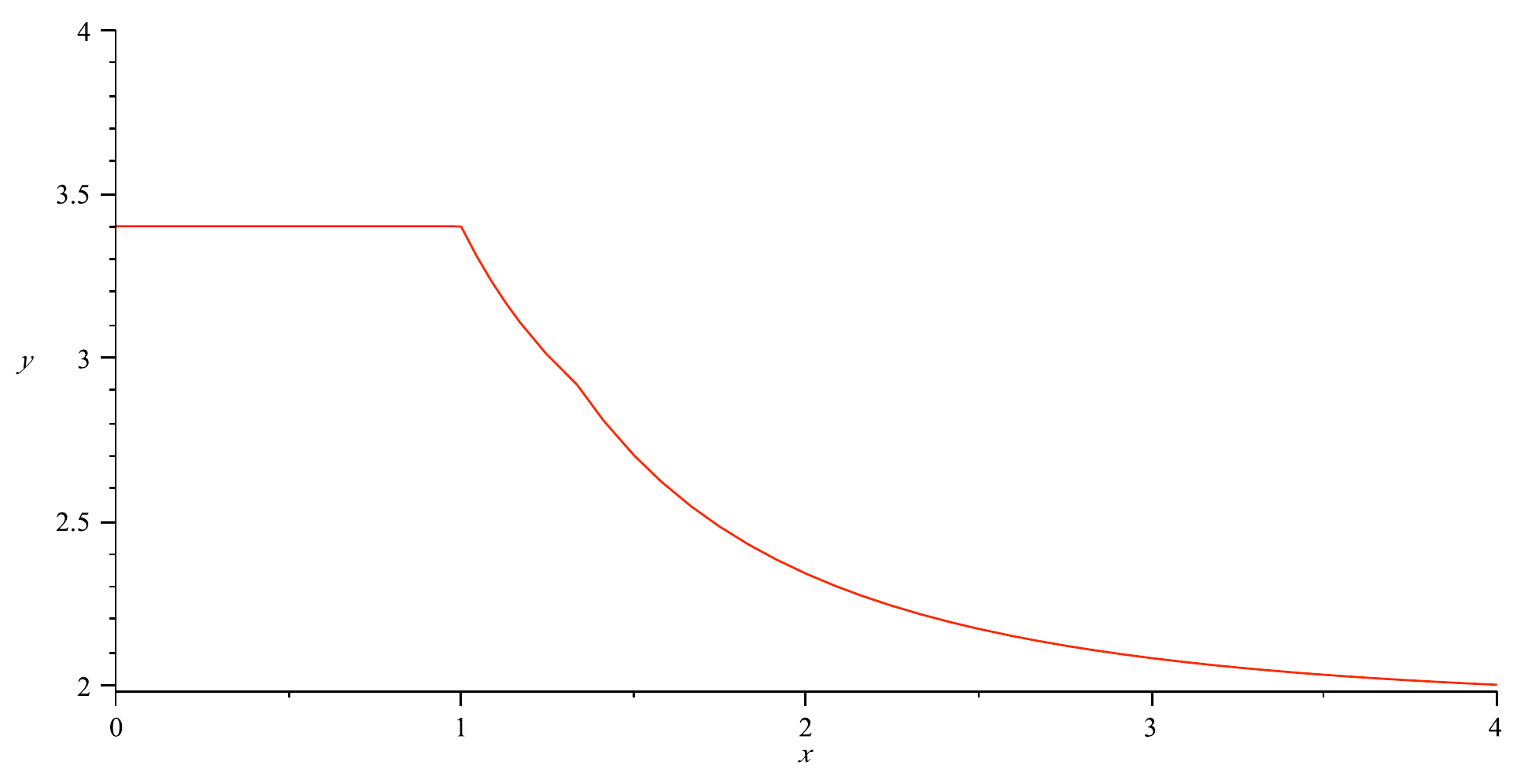}
\caption[]{The graph of $\mu_{7/30}(t)$ assuming Conjecture
  \ref{InfimumRational}}\label{fig:graphs2}
\end{figure}


\section{The topologies induced by the $t$-metric Mahler measures}

In order to proceed with the proof of Theorem \ref{Topologies}, we must recall some definitions and results of \cite{Samuels} and \cite{Samuels3}.
If $S$ is any subset of $\algt$, we write
\begin{equation*}
	\rad(S) = \left\{\alpha\in\algt:\alpha^r\in S\mathrm{\ for\ some}\ r\in\nat\right\}.
\end{equation*}
If $K$ is a number field and $\alpha$ is an algebraic number, let $K_\alpha$ denote the Galois closure of $\rat(\alpha)$ over $\rat$.
We begin with the precise statement of Lemma 3.1 of \cite{Samuels}.

\begin{lem} \label{HeightInK}
	Let $K$ be a Galois extension of $\rat$.  If $\gamma\in\rad(K)$ then there exists a root of unity $\zeta$ and $L,S\in\nat$
	such that $\zeta\gamma^L\in K$ and
	\begin{equation*}
		M(\gamma) = S\cdot M(\zeta\gamma^L).
	\end{equation*}
	In particular, the set
	\begin{equation*}
		\{M(\gamma):\gamma\in\rad(K),\ M(\gamma) \leq B\}
	\end{equation*}
	is finite for every $B \geq 0$.
\end{lem}

It is an easy consequence of Lemma \ref{HeightInK} that $M(\gamma)$ is bounded below by the Mahler measure of an element in $K$.  Indeed, we have that
\begin{equation*}
	M(\gamma) = S\cdot M(\zeta\gamma^L) \geq M(\zeta\gamma^L)
\end{equation*}
and $\zeta\gamma^L\in K$.  Recall that 
\begin{equation*}
	C(\alpha) = \inf\{ M(\gamma): \gamma\in K_\alpha\setminus \tor(\algt)\}
\end{equation*}
and that $C(\alpha) > 0$ by Northcott's Theorem \cite{Northcott}.  We now see easily that
\begin{equation} \label{RadBound}
	M(\gamma) \geq C(\alpha)
\end{equation}
for all $\gamma\in \rad(K_\alpha)\setminus\tor(\algt)$.
We showed in Theorem 1.1 of \cite{Samuels3} that the infimum in $M_t(\alpha)$ is always attained.

\begin{thm} \label{Achieved}
	Suppose $\alpha$ is a non-zero algebraic number and $t\in (0,\infty]$.  Then there exists a point
	\begin{equation*}
		(\alpha_1,\alpha_2,\ldots) \in \tau^{-1}(\alpha) \cap \X(\rad(K_\alpha))
	\end{equation*}
	such that $M_t(\alpha) = \|(M(\alpha_1),M(\alpha_2),\ldots)\|_t$.
\end{thm}

Recalling that $V=\algt/\tor(\algt)$, we may proceed with our proof of Theorem \ref{Topologies}.

\begin{proof}[Proof of Theorem \ref{Topologies}]
	If Lehmer's conjecture is true, then it follows from the results of \cite{FiliSamuels} that $M_\infty$ induces the discrete topology on $V$.
        Furthermore, we always have that $M_t(\alpha) \geq M_\infty(\alpha)$ for all $\alpha\in V$, implying that $M_t$ induces the discrete topology as well,
        establishing one direction of the theorem.
        
        Now assume that Lehmer's conjecture is false and that the topologies induced by $M_t$ and $M_\infty$ are equivalent.  Therefore, the $M_t$ ball
        of radius $1$ centered at $1$,
        \begin{equation*}
          B = \{\bar\gamma\in V: M_t(\bar\gamma) < 1\},
        \end{equation*}
        is open with respect to $M_\infty$.  Therefore, there exists $r > 0$ such that the $M_\infty$-ball
	\begin{equation} \label{Containment}
		B_0 = \{\bar\gamma\in V: M_\infty(\bar\gamma) \leq r\} \subset B.
	\end{equation}
	We have assumed that Lehmer's conjecture is false so there exists a non-trivial point $\bar\alpha\in B_0$.  If $s$ is a positive integer, then the strong triangle 
        inequality implies that $M_\infty(\bar\alpha^s) \leq M_\infty(\bar\alpha) \leq r$ so that $\bar\alpha^s \in B_0$ for all $s\in\nat$.  It follows from \eqref{Containment} that
        \begin{equation} \label{InB}
          \bar\alpha^s\in B
        \end{equation}
        for all $s\in \nat$.  We will now show that $M_t(\bar\alpha^s)$ tends to $\infty$ as $s\to \infty$.

        Select a point $\alpha\in\algt$ whose image in $V$ equals $\bar\alpha$.   In this case, $\alpha$ is not a root of unity.  By
        Theorem \ref{Achieved}, there exists a root of unity $\zeta$ and points
        \begin{equation*}
          \alpha_1,\ldots,\alpha_N \in \rad(K_\alpha)\setminus
          \tor(\algt)
        \end{equation*}
        such that
        \begin{equation*}
          \alpha^s = \zeta\alpha_1\cdots\alpha_N
        \end{equation*}
        and
        \begin{equation} \label{NaiveRep}
          M_t(\alpha^s)^t = \sum_{n=1}^N M(\alpha_n)^t.
        \end{equation}
        
	Recall that the Weil height on $\alpha\in\alg$ is given by
	\begin{equation*}
		h(\alpha) = \frac{M(\alpha)}{\deg\alpha}.
	\end{equation*}
        Using \eqref{NaiveRep}, we have that
        \begin{align*}
          M_t(\alpha^s)^t & =  \sum_{n=1}^N M(\alpha_n)^{t-1}M(\alpha_n) \\
          & \geq \sum_{n=1}^N M(\alpha_n)^{t-1}h(\alpha_n) \\
          & \geq \min_{1\leq n\leq N}\{ M(\alpha_n)\}^{t-1} \cdot \sum_{n=1}^N h(\alpha_n)
        \end{align*}
	It is well-known that the Weil height has the triangle
        inequality $h(\alpha\beta) \leq h(\alpha) + h(\beta)$ as well
        as the identity $h(\alpha) = h(\zeta\alpha)$ for all roots of
        unity $\zeta$.  It follows that
	\begin{align*} \label{SLowerBound}
		M_t(\alpha^s)^t & \geq  \min_{1\leq n\leq N}\{M(\alpha_n)\}^{t-1} \cdot h(\alpha_1\cdots\alpha_N) \\
                & =  \min_{1\leq n\leq N}\{M(\alpha_n)\}^{t-1} \cdot h(\alpha^s)
	\end{align*}
        Furthermore, we have that $h(\alpha^r) = |r|\cdot h(\alpha)$
        for all integers $r$.  This leaves
        \begin{equation} \label{SecondNaiveBound}
          	M_t(\alpha^s)^t \geq s\cdot h(\alpha) \cdot
                \min_{1\leq n\leq N}\{M(\alpha_n)\}^{t-1}
        \end{equation}
        We know that $\alpha$ is not a root of unity so that
        $h(\alpha) > 0$.  Also, We know that $\alpha_n\in\rad(K_\alpha)
        \setminus \tor(\algt)$ for all $n$.  It follows from
        \eqref{RadBound} that $M(\alpha_n) \geq C(\alpha)$ for all
        $n$.  By \eqref{SecondNaiveBound}, we obtain that
        \begin{equation*}
          M_t(\alpha^s)^t \geq s\cdot h(\alpha) \cdot C(\alpha)^{t-1},
        \end{equation*}
        the right hand side of which tends to infinity as $s\to \infty$.  This proves that $\bar\alpha^s\not\in B$ for sufficiently large
        $s$, contradicting \eqref{InB}.
\end{proof}


\section{$\alpha$-standard and $\alpha$-exceptional points} \label{StandardExceptional} 

All of our proofs regarding $\alpha$-standard and $\alpha$-exceptional points are based upon the following result.

\begin{thm} \label{FiniteMin}
	Let $\alpha$ be a non-zero algebraic number and $T$ a positive real number.  Then there exists a finite collection of points
	$\mathcal X = \mathcal X(\alpha, T) \subseteq \bigreal$ such that
	\begin{equation*}
		M_t(\alpha) = \min\{ \|\xx\|_t: \xx\in \mathcal X\}
	\end{equation*}
	for all $t\leq T$.
\end{thm}
\begin{proof}
	By Lemma \ref{HeightInK}, the set
	\begin{equation} \label{RadBounded}
		R(\alpha) = \left\{M(\gamma): \gamma\in\rad(K_\alpha)\ \mathrm{and}\ M(\gamma)\leq M(\alpha)\right\}
	\end{equation}
	is finite and $C(\alpha) = \min R(\alpha)\setminus\tor(\algt)$.  We also note that $M(\alpha) \geq C(\alpha) > 0$.  Next, we define
	\begin{equation*}
		J = J(\alpha,T) = \left\lfloor\left(\frac{M(\alpha)}{C(\alpha)}\right)^T + 1\right\rfloor.
  	\end{equation*}
  	Finally, we write
	\begin{equation*} \label{XDef}
		\mathcal X = \left\{(M(\alpha_1),\ldots,M(\alpha_N),0,0,\ldots): M(\alpha_n)\in R(\alpha),\  N\leq J(\alpha,T)\ \mathrm{and}\ \alpha=\prod_{n=1}^N\alpha_n\right\}.
	\end{equation*}
	We claim that $\mathcal X$ is finite and that
	\begin{equation} \label{MtMin}
		M_t(\alpha) = \min\{ \|\xx\|_t: \xx\in \mathcal X\}
	\end{equation}
	for all $t\leq T$.  We have immediately that $\mathcal X$ injects into
	\begin{equation*}
		\underbrace{R(\alpha) \times \cdots \times R(\alpha)}_{J\ \mathrm{times}}.
	\end{equation*}
	Since each set $R(\alpha)$ is finite, it follows that $\mathcal X$ is finite.
	
	Now we must verify \eqref{MtMin}.  By the definition of $M_t(\alpha)$, we see quickly that
	\begin{equation} \label{PrelimMtMin}
		M_t(\alpha) \leq \min\{ \|\xx\|_t: \xx\in \mathcal X\}.
	\end{equation}
	To show that we always have equality in \eqref{PrelimMtMin}, we must show that, for every positive real $t\leq T$, there exists $\xx\in \mathcal X$ such 
	that $M_t(\alpha) = \|\xx\|_t$.  By Theorem \ref{Achieved}, we know there exist points $\alpha_1,\alpha_2,\ldots,\alpha_{N} \in \rad(K_\alpha)$ 
	such that $\alpha= \alpha_1\cdots\alpha_{N}$ and
  	\begin{equation} \label{tInfimum}
   	 M_t(\alpha) = \|(M(\alpha_1),\ldots,M(\alpha_{N}),0,0,\ldots)\|_t.
  	\end{equation}
 	We may assume without loss of generality that at most one of $\alpha_1,\ldots,\alpha_{N}$ is a root of unity.  Now we write 
  	\begin{equation*}
  		\mm = (M(\alpha_1),\ldots,M(\alpha_{N}),0,0,\ldots)
  	\end{equation*}
  	so we have that
  	\begin{equation*}
  		M_t(\alpha) = \|\mm\|_t.
  	\end{equation*}
  	We must show that $\mm\in \mathcal X$.
  
  By our above remarks, we know that $\alpha_n\in\rad(K_\alpha)$ for all $n$.  Furthermore, we have that $M_t(\alpha) \leq M(\alpha)$, so we also
  obtain that $M(\alpha_n) \leq M(\alpha)$, which implies that $M(\alpha_n)\in R(\alpha)$.
  For every $n$ such that $\alpha_n$ is not a root of unity, we have that $M(\alpha_n) \geq C(\alpha)$ so we obtain
  \begin{equation*}
    M(\alpha)^t \geq M_t(\alpha)^t =  \sum_{n=1}^N M(\alpha_n)^t \geq (N-1)\cdot C(\alpha)^t,
  \end{equation*}
  and therefore,
  \begin{equation*}
    N-1 \leq \left(\frac{M(\alpha)}{C(\alpha)}\right)^t.
  \end{equation*}
  It is clear that $M(\alpha) \geq C(\alpha)$ which yields
  \begin{equation} \label{MaxTerms}
    N\leq  J(\alpha,T)
  \end{equation}
  showing that $\mm\in\mathcal X$ and completing the proof.
\end{proof}

We noted earlier that the continuity of $\mu_\alpha$ was proved in \cite{Samuels3}.  However, Theorem \ref{FiniteMin} gives us a much simpler proof.

\begin{cor} \label{Continuous}
  $\mu_\alpha$ is continuous on $(0,\infty)$.
\end{cor}
\begin{proof}
  On an interval $(0,T]$, Theorem \ref{FiniteMin} establishes that $\mu_\alpha$ is the minimum of a finite number of continuous functions.  It follows that
  $\mu_\alpha$ is itself continuous.
\end{proof}

Before we can prove Theorem \ref{FiniteBelow2}, we give one additional definition along with a lemma. 
For a positive real number $T$ and an algebraic number $\alpha$, we will, for the remainder of this paper, let
$\mathcal X = \mathcal X(\alpha,T)$ be as in the conclusion of Theorem \ref{FiniteMin}.  We say that $s\leq T$ is an {\it intersection point with 
respect to $\mathcal X$} if there exist 
$\xx, \yy\in \mathcal X$ such that $\|\xx\|_s = \|\yy\|_s$ but $t\mapsto \|\xx\|_t$ is not the same function as $t\mapsto \|\yy\|_t$.

\begin{lem} \label{NonIntersection}
	Suppose that $\alpha$ is a non-zero algebraic number and $T$ is a positive real number.  If $I\subseteq(0,T]$ is an interval containing no intersection points
	with respect to $\mathcal X(\alpha,T)$ then $I$ is $\alpha$-uniform.
\end{lem}
\begin{proof}
	Assume that $I$ is not $\alpha$-uniform and fix a point $t\in I$.  By definition of $\alpha$-uniform, for every point $\xx\in \mathcal X$ such that 
	$M_t(\alpha) = \|\xx\|_t$, there exists $s\in I$ such that $M_s(\alpha) < \|\xx\|_s$.  We may select $\yy\in \mathcal X$ such that $M_s(\alpha) = \|\yy\|_s$ 
	and note that $M_t(\alpha) \leq \|\yy\|_t$.  Hence, we have that
	\begin{equation*}
		\|\xx\|_t \leq \|\yy\|_t\ \mathrm{and}\ \|\xx\|_s > \|\yy\|_s.
	\end{equation*}
	By the Intermediate Value Theorem, there exists a point $r$ between $s$ and $t$ such that $\|\xx\|_r = \|\yy\|_r$.  This means that $I$ contains
	an intersection point, a contradiction.
\end{proof}

We are now prepared to prove Theorem \ref{FiniteBelow2}.

\begin{proof}[Proof of Theorem \ref{FiniteBelow2}]
	We first show that there are only finitely many intersection points of $\mathcal X$.  Let 
	\begin{equation*}
		\xx = (x_1,\ldots,x_N,0,0,\ldots)\ \mathrm{and}\ \yy = (y_1,\ldots,y_M,0,0,\ldots)
	\end{equation*}
	be elements of $\mathcal X$ such that $x_n,y_n \geq 0$.  Further suppose that $t\mapsto \|\xx\|_t$ and $t\mapsto \|\yy\|_t$ are distinct functions.  Now write
	\begin{equation*}
		F(z) = \sum_{n=1}^N x_n^z - \sum_{m=1}^M y_m^z
	\end{equation*}
	and note that $F(z)$ is an entire function with $F\not\equiv 0$.  If $F$ has infinitely many zeros $[0,T]$,
	then these zeros have a cluster point in $\com$, a contradiction.  So $F$ may only have finitely many zeros in $[0,T]$, and hence, the functions
	$\|\xx\|_t$ and $\|\yy\|_t$ may only intersect in finitely many points in $[0,T]$.  It now follows that there are only finitely many intersection points.
	
	Next, assume that $t$ is not an intersection point.  Since the set of intersection points is finite, we know there exists a neighborhood $I$ of 
	$t$ that contains no intersection points.  It now follows from Lemma \ref{NonIntersection} that $I$ is $\alpha$-uniform so that $t$ is $\alpha$-standard.
	In other words, we have shown that every $\alpha$-exceptional point in $(0,T)$ must also be an intersection point.  However, there are only finitely many 
	intersection points, so there are only finitely many $\alpha$-exceptional points in $(0,T)$.
\end{proof}

We now proceed with the proof of Theorem \ref{UniformIntervals}, which requires the following two lemmas.
The first of these lemmas shows that even $\alpha$-exceptional points have neighborhoods that are relatively well behaved.

\begin{lem} \label{Neighborhood}
	If $t\in (0,\infty)$ then there exists a neighborhood $(a,b)$ of $t$ such that $(a,t]$ and $[t,b)$ are $\alpha$-uniform.
\end{lem}
\begin{proof}
	If $t$ is $\alpha$-standard, then the result is obvious, so we may assume that $t$ is $\alpha$-exceptional.
	
	Set $T=t+1$ and let $\mathcal X = \mathcal X(\alpha,T)$ be the set from the conclusion of Theorem \ref{FiniteMin}.  Since $\mathcal X$ has only finitely
	many intersection points, there must exist a neighborhood $(a,b)$ of $t$ containing no intersection points except $t$.  In particular, $(t,b)$ contains
	no intersection points, so it follows from Lemma \ref{NonIntersection} that $(t,b)$ is $\alpha$-uniform.  Therefore, there exists $\xx\in\mathcal X$ such that
	$M_s(\alpha) = \|\xx\|_s$ for all $s\in(t,b)$.
	
	By Theorem \ref{Continuous}, we know that $\mu_\alpha$ is continuous on $[t,b)$.  Of course, $s\mapsto \|\xx\|_s$ is also continuous on this interval so that
	\begin{equation*}
		M_t(\alpha) = \lim_{s\to t^+} M_s(\alpha) = \lim_{s\to t^+} \|\xx\|_s = \|\xx\|_t
	\end{equation*}
	showing that $M_s(\alpha) = \|\xx\|_s$ for all $s\in [t,b)$.  This establishes that $[t,b)$ is $\alpha$-uniform.  A similar argument is used to show that
	$(a,t]$ is $\alpha$-uniform, completing the proof.
\end{proof}

Our next lemma shows that, in order to prove that an interval $I$ is $\alpha$-uniform, we need only show the existence of a cover of $I$ by
$\alpha$-uniform open intervals.  Here, we understand that open means open with respect to $S$.

\begin{lem} \label{CoverLemma}
  Suppose $S\subset (0,\infty)$ is any interval.  If there exists a finite cover of $S$ by $\alpha$-uniform open intervals, then $S$
  is $\alpha$-uniform.
\end{lem}
\begin{proof}
  Suppose $\{I_n:1\leq n\leq N\}$ is a collection of open intervals in $(0,\infty)$ such that
	\begin{equation*}
		S = \bigcup_{n=1}^N I_n
	\end{equation*}
	and $I_n$ is $\alpha$-uniform for all $n$.  Since $S$ is connected, we must have that
	\begin{equation*}
		I_1 \bigcap \left(\bigcup_{n=2}^N I_n\right) \ne \emptyset
	\end{equation*}
	so that there exists some $k$ such that $I_1\cap I_k \ne \emptyset$.  Since both $I_1$ and $I_k$ are open intervals,
	their intersection must be a non-empty open interval.  We know that $I_1$ and $I_k$ are $\alpha$-uniform, so there exist points
	$(x_1,\ldots,x_L,0,0,\ldots), (y_1,\ldots,y_M,0,0\ldots) \in \bigreal$ such that
	\begin{equation*}
		M_t(\alpha)^t = \sum_{l=1}^L x_l^t\quad\mathrm{for\ all}\ t\in I_1
	\end{equation*}
	and
	\begin{equation*}
		M_t(\alpha)^t = \sum_{m=1}^M y_m^t\quad\mathrm{for\ all}\ t\in I_k.
	\end{equation*}
	These functions must be equal on the open interval $I_1\cap I_k$.  That is, we have that
	\begin{equation} \label{EntireEqual}
		\sum_{l=1}^L x_l^z = \sum_{m=1}^M y_m^z
	\end{equation}
	on a set having a limit point in $\com$.  Since both sides of \eqref{EntireEqual} are entire functions, we conclude that they must be equal on all of $\com$.
	In particular, we have shown that
	\begin{equation*}
		M_t(\alpha)^t = \sum_{l=1}^L x_l^t\quad\mathrm{for\ all}\ t\in I_1\cup I_k
	\end{equation*}
	implying that $I_1\cup I_k$ is $\alpha$-uniform.  We now see that the set of intervals
	\begin{equation*}
		\{I_1\cup I_k\} \cup \{I_n: 2\leq n\leq N\ \mathrm{and}\ n\ne k\}
	\end{equation*}
	is a cover of $S$ by $N-1$ $\alpha$-uniform open intervals.  Repeating the above argument $N-1$ more times, we obtain a cover containing only one interval.
\end{proof}

In view of the above lemmas, the proof of Theorem \ref{UniformIntervals} is fairly straightforward.

\begin{proof}[Proof of Theorem \ref{UniformIntervals}]
  If $[a,b]$ is $\alpha$-uniform, then it is clear that every point in $(a,b)$ is $\alpha$-standard.  Similarly, if $(0,a]$ is $\alpha$-uniform
  then every point in $(0,a)$ is $\alpha$-standard.  We now prove the opposite directions of both statements beginning with the first.

  Assume now that every point in $(a,b)$ is $\alpha$-standard.  Hence, there exists a cover $\mathcal I$ of $(a,b)$ by $\alpha$-uniform
  open intervals.  Furthermore, by Lemma \ref{Neighborhood}, there exist points $c, d\in (a,b)$ such that the intervals
   \begin{equation*}
     J_1 = [a,c)\quad\mathrm{and}\quad J_2 = (d,b]
   \end{equation*}
   are $\alpha$-uniform.  Therefore, the collection of intervals
   \begin{equation*}
     \{J_1\} \cup \{J_2\} \cup \mathcal I
   \end{equation*}
   forms a cover of $[a,b]$ by $\alpha$-uniform intervals which are all open with respect to $[a,b]$.  Since $[a,b]$ is compact there exists
   a finite subcover and the result follows from Lemma \ref{CoverLemma}.

   To prove the second statement, recall that \cite{Samuels3} establishes $\mu_\alpha$ to be constant in a neighborood of $0$.  In particular, there exists $\varepsilon >0$
   such that $(0,2\varepsilon)$ is $\alpha$-uniform.  We know that $(0,a)$ contains no $\alpha$-standard points, so that $(\varepsilon,a)$ does not either.  By the first
   statement of this theorem, we know that $(\varepsilon,a]$ is $\alpha$-uniform.  Certainly
   \begin{equation*}
     (0,2\varepsilon) \cup (\varepsilon,a]
   \end{equation*}
   is a finite cover of $(0,a]$ by $\alpha$-uniform intervals that are open in $(0,a]$.  It follows from Lemma \ref{CoverLemma} that $(0,a]$ is $\alpha$-uniform.
\end{proof}

Equipped with Theorems \ref{FiniteBelow2} and \ref{UniformIntervals}, we can give our proof of Corollary \ref{NonOver}.

\begin{proof}[Proof of Corollary \ref{NonOver}]
  By Theorem \ref{FiniteBelow2}, there are finitely many exceptional points in $(0,T)$.  Suppose these points are given by
  \begin{equation*}
    0 < t_1 < t_2 < \cdots < t_N < T.
  \end{equation*}
  We write $I_0 = (0,t_1]$, $I_N = [t_N,T]$ and $I_n = [t_n,t_{n+1}]$
  for all other values of $n$.  We write
  \begin{equation*}
    \mathcal I = \bigcup_{n=0}^N \{I_n\}
  \end{equation*}
  and claim that $\mathcal I$ satisfies the required properties.
  Clearly, $\mathcal I$ is a finite set of non-overlapping closed
  intervals with
  \begin{equation*}
    (0,T] = \bigcup_{I\in\mathcal I} I,
  \end{equation*}
  which establishes \eqref{Union}.
  The interior of $I_n$ contains only $\alpha$-standard points, so by Theorem \ref{UniformIntervals}, $I_n$ is
  $\alpha$-uniform for all $n$, verifying \eqref{Uniforms}.  
  
  Now assume that $t\in (0,T)$ is $\alpha$-exceptional.  By \eqref{Uniforms}, $t$ must lie at an endpoint of an inteval $I\in\mathcal I$, 
  so that $t$ must lie at point where two intervals from $\mathcal I$ intersect.
  If $t\in [t_{n-1},t_n]\cap [t_n,t_{n+1}]$, then $t = t_n$ implying that $t$ is $\alpha$-exceptional and verifying \eqref{Exceptionals}.
\end{proof}

Finally, we may proceed with the proof of Theorem \ref{StandardDiff}.

\begin{proof}[Proof of Theorem \ref{StandardDiff}]
  If $s$ is $\alpha$-standard, then there exists $\xx\in\mathcal X$ and a neigborhood $I$ of $s$ such that
  \begin{equation} \label{AlphaStandardNear}
    M_t(\alpha) = \|\xx\|_t
  \end{equation}
  for all $t\in I$.  Certainly, the right hand side of \eqref{AlphaStandardNear} is infinitely differentiable as a function of $t$ for all positive $t$.
  
  Assume now that $\mu_\alpha$ is infinitely differentiable at $s$.  By Lemma \ref{Neighborhood}, there exists a neighborhood $(a,b)$ of $s$ such that
  $(a,s]$ and $[s,b)$ are $\alpha$-uniform.  Suppose that $\xx, \yy \in\mathcal X$ are such that $M_t(\alpha) = \|\xx\|_t$ for all $t\in (a,s]$ and
  $M_t(\alpha) = \|\yy\|_t$ for all $t\in [s,b)$.  Now write
  \begin{equation*}
  	f(z) = \|\xx\|_z^z\quad\mathrm{and}\quad g(z) = \|\yy\|_z^z
  \end{equation*}
  and observe that $f$ and $g$ are entire functions.  Moreover, their Taylor series expansions at $s$, given by
  \begin{equation} \label{TaylorExpansions}
  		f(z) = \sum_{n=0}^\infty \frac{f^{(n)}(s)}{n!}(z-s)^n\qquad\mathrm{and}\qquad g(z) = \sum_{n=0}^\infty \frac{g^{(n)}(s)}{n!}(z-s)^n,
  \end{equation}
  converge in all of $\com$.  
  
  For the remainder of this proof, we will write $\ell(t) = \mu_\alpha(t)^t$.  By our assumption, $\ell$ is infinitely differentiable at $s$.  We also have that
  $\ell(t) = f(t)$ for all $t\in(a,s]$ which implies that $\ell$ must also be infinitely differentiable in $(a,s)$.  It follows easily that
  \begin{equation} \label{EqualDerivs}
  	\ell^{(n)}(t) = f^{(n)}(t)\quad\mathrm{for\ all}\ t\in (a,s).
  \end{equation}
  
 We now prove by induction that $f^{(n)}(s) = \ell^{(n)}(s)$.  By the definitions of our functions, we obtain immediately $f(s) = \ell(s)$ establishing the base case.
  Assuming now that $f^{(n)}(s) = \ell^{(n)}(s)$, we may write
  \begin{equation*}
  	\ell^{(n+1)}(s) = \lim_{h\to 0}\frac{\ell^{(n)}(s+h) - \ell^{(n)}(s)}{h} = \lim_{h\to 0^-}\frac{\ell^{(n)}(s+h) - f^{(n)}(s)}{h}.
  \end{equation*}
  However, using \eqref{EqualDerivs}, it follows that $f^{(n)}(s+h) = \ell^{(n)}(s+h)$ for $h$ sufficiently close to $0$.  We now have that
  \begin{equation} \label{FinalDeriv}
  	\ell^{(n+1)}(s) = \lim_{h\to 0^-}\frac{f^{(n)}(s+h) - f^{(n)}(s)}{h}.
  \end{equation}
  We know already that $f$ must be infinitely differentiable at $s$, so that the right hand side of \eqref{FinalDeriv} must equal $f^{(n+1)}(s)$
  establishing our claim that $f^{(n)}(s) = \ell^{(n)}(s)$ for all $n$.
  
  A similar argument can be used to show that $g^{(n)}(s) = \ell^{(n)}(s)$, and therefore $g^{(n)}(s) = f^{(n)}(s)$.  It now follows from \eqref{TaylorExpansions}
  that $f(z) = g(z)$ for all $z\in \com$.  In particular, we have shown that
  \begin{equation*}
  	\mu_\alpha(t) = \|\xx\|_t = \|\yy\|_t
  \end{equation*}
  for all $t\in (a,b)$, proving that $(a,b)$ is $\alpha$-uniform.  It follows that $s$ is $\alpha$-standard.
\end{proof}



\section{Proof of Theorem \ref{RationalFactorization}}

\begin{proof}[Proof of Theorem \ref{RationalFactorization}]
Since we are assuming Conjecture \ref{InfimumRational}, we have that there exist positive integers $r_1,\ldots,r_N,s_1,\ldots s_N$ such that
\begin{equation} \label{RepAsRational}
	\frac{r}{s} = \prod_{n=1}^N \frac{r_n}{s_n}
\end{equation}
and
\begin{equation} \label{FirstRationalRep}
	M_t\left(\frac{r}{s}\right)^t = \sum_{n=1}^N M\left(\frac{r_n}{s_n}\right)^t = \sum_{n=1}^N \max\{|r_n|,|s_n|\}^t
\end{equation}
Suppose that $\gcd(r_i,s_j)>1$ for some $i$ and $j$ so there exists a prime number $p$ such that $p\mid r_i$ and $p\mid s_j$.  Now define points $r'_n$
and $s'_n$, for $1\leq n\leq N$, by
\begin{equation*}
	r'_n = \left\{ \begin{array}{ll}
		r_n & \mathrm{if}\ n\ne i \\
		r_n/p & \mathrm{if}\ n = i
		\end{array}\right.
\end{equation*}
and
\begin{equation*}
	s'_n = \left\{ \begin{array}{ll}
		s_n & \mathrm{if}\ n\ne j \\
		s_n/p & \mathrm{if}\ n = j.
		\end{array}\right.
\end{equation*}
We note immediately that
\begin{equation*}
	\frac{r}{s} = \prod_{n=1}^N \frac{r'_n}{s'_n}
\end{equation*}
and
\begin{equation*} \label{NewBounds}
	\max\{|r'_n|,|s'_n|\} \leq \max\{|r_n|,|s_n|\}
\end{equation*}
for all $n$.  Then using \eqref{FirstRationalRep}, we find that
\begin{align*}
	M_t\left(\frac{r}{s}\right)^t & \leq \sum_{n=1}^N M\left(\frac{r'_n}{s'_n}\right)^t \\
		& =  \sum_{n=1}^N \max\{|r'_n|,|s'_n|\}^t \\
		& \leq \sum_{n=1}^N \max\{|r_n|,|s_n|\}^t \\
		& = M_t\left(\frac{r}{s}\right)^t 
\end{align*}
implying that
\begin{equation*}
	M_t\left(\frac{r}{s}\right)^t = \sum_{n=1}^N M\left(\frac{r'_n}{s'_n}\right)^t 
\end{equation*}
Repeating this process, we can find positive integers $a_1,\ldots,a_N,b_1,\ldots,b_N$ such that
\begin{equation*}
	M_t\left(\frac{r}{s}\right)^t = \sum_{n=1}^N M\left(\frac{a_n}{b_n}\right)^t
\end{equation*}
and each pair $(a_i,b_j)$ are relatively prime.  In particular, we have that
\begin{equation} \label{PrimeProducts}
	\gcd\left( \prod_{n=1}^Na_n,\prod_{n=1}^Nb_n\right) = 1.
\end{equation}

By \eqref{RepAsRational}, we have that
\begin{equation*}
	r\prod_{n=1}^N b_n = s\prod_{n=1}^N a_n.
\end{equation*}
This means that $r \mid s\prod_{n=1}^N a_n$, but since $\gcd(r,s) = 1$, we have that
\begin{equation} \label{FirstDivide}
	r\mid \prod_{n=1}^N a_n.
\end{equation}
However, we also know that $\prod_{n=1}^n a_n \mid r\prod_{n=1}^N b_n$, so that by \eqref{PrimeProducts}, we obtain
\begin{equation*}
	\prod_{n=1}^N a_n \mid r.
\end{equation*}
Combining this with \eqref{FirstDivide}, we find that
\begin{equation*}
	\prod_{n=1}^N a_n = r.
\end{equation*}
A similar argument can be used to prove that $\prod_{n=1}^N b_n = s$ which completes the proof.

\end{proof}

Finally, we provide our proof of Theorem \ref{Integers}.

\begin{proof}[Proof of Theorem \ref{Integers}]
  First assume that $t\leq 1$.  It was shown in \cite{DubSmyth2} that
  $M_1(\alpha) = M(\alpha)$ whenever $\alpha$ is rational.  Using the fact that $\mu_\alpha$ is decreasing, we
  have that
  \begin{equation*}
    M(\alpha) = M_1(\alpha) \leq M_t(\alpha) \leq M(\alpha).
  \end{equation*}
  But $M(\alpha) = \log \alpha$ so the result follows for $t\leq 1$.

  Now suppose that $t>1$.  By Theorem \ref{RationalFactorization}, there exist integers
  $k_1,\ldots k_N$ such that $\alpha = k_1\cdots k_n$ and
  \begin{equation} \label{StartingFactorization}
    M_t(\alpha)^t = \sum_{n=1}^N M(k_n)^t = \sum_{n=1}^N (\log k_n)^t.
  \end{equation}
  We claim that each $k_n$ must be prime.  To see this, assume there
  exists an integer $j$ such that $k_j$ is not prime and write
  \begin{equation*}
    k_j = ab
  \end{equation*}
  where $a,b\in \nat$ and $a,b > 1$.  It is a straightforward application of the Mean Value Theorem to show that
  \begin{equation*}
    (\log k_j)^t = (\log a +\log b)^t > (\log a)^t + (\log b)^t.
  \end{equation*}
  Applying \eqref{StartingFactorization}, we find that
  \begin{equation} \label{NextFactorization}
    M_t(\alpha)^t > (\log a)^t + (\log b)^t + \sum_{\substack{n=1\\ n\ne j}}^N (\log k_n)^t.
  \end{equation}
  However, we also have that
  \begin{equation*}
    \alpha = ab\cdot\prod_{\substack{n=1\\ n\ne j}}^N k_n
  \end{equation*}
  which yields immediately
  \begin{equation*}
    M_t(\alpha)^t \leq M(a)^t + M(b)^t + \sum_{\substack{n=1\\ n\ne j}}^N
    M(k_n)^t =  (\log a)^t + (\log b)^t + \sum_{\substack{n=1\\ n\ne j}}^N (\log k_n)^t
  \end{equation*}
  contradicting \eqref{NextFactorization}.  We have now shown that
  each $k_n$ must be prime completing the proof.
\end{proof}

\end{document}